\documentclass[11pt,reqno]{amsart}
\usepackage{amssymb}
\usepackage{natbib}
\bibpunct[, ]{[}{]}{;}{n}{,}{,}
\numberwithin{equation}{section}
\allowdisplaybreaks

\newtheorem{theorem}{Theorem}[section]
\newtheorem{lemma}[theorem]{Lemma}

\newtheorem{condition}[theorem]{Condition}

\theoremstyle{definition}

\newtheorem{remark}[theorem]{Remark}

\newtheorem*{acks}{Acknowledgements}

\theoremstyle{remark}

\newenvironment{romenumerate}{\begin{enumerate}
 }{\end{enumerate}}

\newcounter{thmenumerate}

\newcounter{xenumerate}

\newcommand{\refT}[1]{Theorem~\ref{#1}}

\newcommand{\refCN}[1]{Condition~\ref{#1}}
\newcommand{\refL}[1]{Lemma~\ref{#1}}
\newcommand{\refR}[1]{Remark~\ref{#1}}
\newcommand{\refS}[1]{Section~\ref{#1}}

\newcommand{\refand}[2]{\ref{#1} and~\ref{#2}}

\begingroup
  \divide\count255 by 60
  \multiply\count255 by -60
  \advance\count255 by \time
  \ifnum \count255 < 10 \xdef\klockan{\the\count1.0\the\count255}
  \else\xdef\klockan{\the\count1.\the\count255}\fi
\endgroup

\newcommand{\sumji}{\sum_{j=1}^\infty}
\newcommand{\sumk}{\sum_{k=0}^\infty}
\newcommand{\sumki}{\sum_{k=1}^\infty}

\newcommand{\sumin}{\sum_{i=1}^n}

\newcommand\set[1]{\ensuremath{\{#1\}}}

\newcommand\xpar[1]{(#1)}
\newcommand\bigpar[1]{\bigl(#1\bigr)}
\newcommand\Bigpar[1]{\Bigl(#1\Bigr)}

\newcommand\lrpar[1]{\left(#1\right)}

\newcommand\bigabs[1]{\bigl|#1\bigr|}

\newcommand\biggabs[1]{\biggl|#1\biggr|}
\newcommand\lrabs[1]{\left|#1\right|}
\def\rompar(#1){\textup(#1\textup)}    

\def\xexp(#1){e^{#1}}

\newcommand\floor[1]{\lfloor#1\rfloor}

\newcommand\ntoo{\ensuremath{{n\to\infty}}}

\newcommand\iid{i.i.d.\spacefactor=1000}
\newcommand\ie{i.e.\spacefactor=1000}
\newcommand\eg{e.g.\spacefactor=1000}

\newcommand\cf{cf.\spacefactor=1000}

\newcommand{\tend}{\longrightarrow}
\newcommand\dto{\overset{\mathrm{d}}{\tend}}
\newcommand\pto{\overset{\mathrm{p}}{\tend}}

\newcommand\Op{O_{\mathrm p}}
\newcommand\op{o_{\mathrm p}}

\newcommand\bbN{\mathbb N}

\newcommand\E{\operatorname{\mathbb E{}}}
\renewcommand\P{\operatorname{\mathbb P{}}}
\newcommand\Var{\operatorname{Var}}

\newcommand\Exp{\operatorname{Exp}}
\newcommand\Po{\operatorname{Po}}
\newcommand\Bi{\operatorname{Bi}}

\newcommand\ESF{\operatorname{ESF}}

\newcommand\ga{\alpha}
\newcommand\gb{\beta}
\newcommand\gd{\delta}

\newcommand\gam{\gamma}

\newcommand\gl{\lambda}

\newcommand\go{\omega}

\newcommand\eps{\varepsilon}

\renewcommand\phi{\varphi}

\newcommand\cC{\mathcal C}

\newcommand\cE{\mathcal E}

\newcommand\tA{{\widetilde A}}
\newcommand\tS{{\widetilde S}}
\newcommand\tV{{\widetilde V}}
\newcommand\tH{\widetilde{H}}

\def\[#1]{[\![#1]\!]}

\newcommand\qq{^{1/2}}
\newcommand\qqw{^{-1/2}}
\newcommand\qqq{^{1/3}}
\newcommand\qqqw{^{-1/3}}
\newcommand\qw{^{-1}}
\newcommand\qww{^{-2}}
\newcommand\qqcw{^{-3/2}}

\renewcommand{\=}{:=}

\newcommand\gnp{\ensuremath{G(n,p)}}
\newcommand\gnm{\ensuremath{G(n,m)}}

\newcommand\gnd{\ensuremath{G(n,(d_i)_1^n)}}
\newcommand\gndx{\ensuremath{G^*(n,(d_i)_1^n)}}

\newcommand\nn{^{(n)}}

\newcommand\nxf{N^{(\floor x)}}
\newcommand\nx{N^{( x)}}

\newcommand\whp{{whp}}

\newenvironment{Cenumerate}{\begin{enumerate}
 }{\end{enumerate}}
\newcommand\suptoo{\sup_{t\ge0}}
\newcommand\ee[1]{e^{#1}}
\newcommand\eez[1]{e^{-#1}}
\newcommand\eet{e^{-t}}
\newcommand\tvk{\tV_k}
\newcommand\dmax{d_\textup{max}}
\newcommand\opn{\op(n)}

\newcommand\tC{{\widetilde \cC}}
\newcommand\pC{{\cC'}}
\newcommand\ppC{{\cC''}}
\newcommand\refCNN{\refCN{C1}}
\newcommand\refCC[1]{\refCN{C1}\ref{#1}}
\newcommand\gan{\ga_n}
\newcommand\gbn{\gb_n}
\newcommand\suptto{\sup_{t\le t_0}}
\newcommand\suptanto{\sup_{t\le \gan t_0}}
\newcommand\hhn{\check{H}_n}
\newcommand\thalf{\tfrac12}
\newcommand\half{\frac12}
\newcommand\nkn{\frac{n_k}n}
\newcommand\ngaaa{n\gan^3}
\newcommand\eee{\cE_\eps}
\newcommand\epsn{\eps_n}

\newcommand\CS{Cauchy--Schwarz}

\begin{document}
\title
{A new approach to the giant component problem}

\date{12 July 2007} 

\author{Svante Janson}
\address{Department of Mathematics, Uppsala University, PO Box 480,
SE-751~06 Uppsala, Sweden}
\email{svante.janson@math.uu.se}
\urladdr{http://www.math.uu.se/\~{}svante/}

\author{Malwina J. Luczak}
\address{Department of Mathematics, London School of Economics,
  Houghton Street, London WC2A 2AE, United Kingdom}
\email{m.j.luczak@lse.ac.uk}
\urladdr{http://www.lse.ac.uk/people/m.j.luczak@lse.ac.uk/}


\begin{abstract}
We study the largest component of a random (multi)graph on $n$
vertices with a given degree sequence.
We let \ntoo. Then, under some regularity conditions on the degree
sequences,
we give conditions on the asymptotic shape of the degree sequence
that imply that
with high probability all the components are small,
and other conditions that imply that with high probability
there is a giant component and the sizes of its vertex and edge sets
satisfy a law of large numbers;
under suitable assumptions these are the only two possibilities.
In particular, we recover the results by Molloy and Reed
\cite{MR95,MR98} on the size of the largest component in a random
graph with a given degree sequence.

We further obtain a new sharp result for the giant component just
above the threshold,
generalizing the case of $\gnp$ with $np=1+\omega(n)n\qqqw$, where
$\go(n)\to\infty$ arbitrarily slowly.

Our method is based on the properties
of empirical distributions of independent random variables, and leads
to simple proofs.
\end{abstract}

\maketitle

\section{Introduction}\label{S:intro}

For many years, questions concerning the size and structure of the
largest component
in a random graph have attracted a lot of attention. There have
by now been quite a number of studies for the Bernoulli random graph $\gnp$
with $n$ vertices and edge
probability $p$, and for the uniformly random graph $\gnm$ with $n$
vertices and $m$ edges (see for instance~\cite{b01,JLR} and the references
therein).
Further,
a number of
studies~\cite{MR95,MR98,KS06} have considered the emergence of a giant
component in a random graph with a specified degree sequence.
In~\cite{MR95}, Molloy and Reed found
the threshold for the appearance of a giant component in a random
graph on $n$ vertices with a given degree sequence;
in~\cite{MR98}, they gave further results including the size of this
giant component above this critical window.
Their strategy was to analyse an edge deletion algorithm that finds
the components in a graph, showing that the corresponding random process
is well approximated by the solution to a system of differential equations.
The proof is rather long and complicated, and uses a bound of the
order $n^{1/4}$ on the maximum degree.
More recently, \citet{KS06}
have considered the near-critical behaviour of such graphs, once
again assuming that, for some $\epsilon >0$, the maximum degree does
not exceed $n^{1/4 - \epsilon}$. Using singularity analysis of
generating functions, they determine the size of the giant component
very close to the critical window, with a gap logarithmic in the
number of vertices.

In this paper, we present a simple solution to the giant component
problem. Unlike~\citet{MR95,MR98}, we do not use
differential equations, but rely solely
on the convergence of empirical distributions of independent
random variables.
(We use a variant of the method we used
in \cite{SJ184,SJ196}
to study the $k$-core of a
random graph.)
In the super-critical regime, we require only
conditions on the second moment of the asymptotic degree distribution;
in the critical regime, we require a fourth moment condition, but we
are able to go all the way to the critical window, without any
logarithmic separation. This is striking, as that logarithmic (or even larger)
separation is often very hard to get rid of, see for
instance~\cite{HL06} in the case of percolation on the Cartesian
product of two complete graphs on $n$ vertices, or~\cite{BCHSS06} in percolation on the $n$-cube, and also~\cite{KS06}
for the model analysed in the present paper.
Like~\citet{MR95,MR98}, we work
directly in the configuration model used to construct the random
graph, exposing the edges one by one as they are needed.

We work with random graphs with given vertex degrees. Results for some
other random graph models, notably for \gnp{} and \gnm, follow
immediately by conditioning on the vertex degrees.

Our method uses a version of the standard exploration of
components.
A commonly used, and very successful,  method to study the giant
component is to make a branching process approximation of the early
stages of this exploration, thus focussing on the
beginning of the exploration of each component
and the conditions for not becoming extinct too soon; see \eg{}
\citet{JLR,MR95,KS06}
and, for some more complicated cases, \citet{SJ199}.
It should be noted that, in contrast, our method focuses on the
condition for \emph{ending} each exploration.

\section{Notation and results}

To state our results we introduce some notation.
For a graph $G$, let
$v(G)$ and $e(G)$ denote the numbers of  vertices and edges in $G$,
respectively; further, let $v_k(G)$ be the number of vertices of
degree $k$, $k\ge0$.

Let $n \in \bbN$ and let $(d_i)_1^n$ be a sequence of
non-negative integers.
We let \gnd{} be a random graph with degree sequence
$(d_i)_1^n$, uniformly chosen among all possibilities (tacitly
assuming that there is any such graph at all).

It will be convenient in the proofs below
to work with \emph{multigraphs}, that is to allow
multiple edges and loops.
More precisely,
we shall use the following standard type of random multigraph:
Let $n \in \bbN$ and let $(d_i)_1^n$ be a sequence of
non-negative integers such that  $\sumin d_i$ is even.
We let \gndx{} be the \emph{random multigraph with given degree
sequence $(d_i)_1^n$}, defined
by the configuration model (see \eg{} \citet{b01}):
take a set of $d_i$ \emph{half-edges} for each vertex
$i$, and combine the half-edges into pairs by a uniformly random
matching of the set of all half-edges.
Note that \gndx{} does not have exactly the uniform
distribution over all multigraphs with the given degree sequence;
there is a weight with a factor $1/j!$ for every edge of multiplicity $j$, and
a factor $1/2$ for every loop, see \eg{} \cite[\S1]{giant}. However,
conditioned on the multigraph being a (simple) graph, we obtain
\gnd,
the
uniformly distributed random graph with the given degree sequence.

We assume throughout the paper that we are given a sequence
$(d_i)_1^n=(d_i\nn)_1^n$
for each $n\in\bbN$ (or at least for some sequence \ntoo);
for notational simplicity we will usually not show the dependence on
$n$ explicitly.
We consider asymptotics as \ntoo, and all unspecified limits below are
as \ntoo.
We say that an event holds \whp{} (\emph{with high probability}),
if it holds with probability tending to 1 as
$n\to\infty$.
We shall use $\pto$ for convergence in probability and
$\Op$ and
$\op$ in the standard way (see \eg{}
Janson, {\L}uczak and Ruci\'nski~\cite{JLR}); for example,
if $(X_n)$ is a sequence of random variables, then
$X_n=\Op(1)$ means ``$X_n$ is bounded in probability'' and
$X_n=\op(1)$ means that $X_n \pto0$.

We write
\begin{align*}
m&=m(n)\=\tfrac12\sumin d_i
\intertext{and}
n_k&=n_k(n)\=\#\set{i:d_i=k},
\quad k\ge0;    
\end{align*}
thus $m$ is the number of
edges and $n_k$ is the number of vertices of degree
in
the random graph \gnd{} (or \gndx).
We assume that the given
$(d_i)_1^n$ satisfy the following regularity conditions,
\cf{} Molloy and Reed \cite{MR95,MR98} (where similar but not
identical conditions are assumed).

\begin{condition}\label{C1}
For each $n$,  $(d_i)_1^n=(d_i\nn)_1^n$ is a sequence of non-negative
integers such that $\sumin d_i$ is even.
Furthermore,
$(p_k)_{k=0}^\infty$ is a probability distribution independent of $n$
such that
\begin{romenumerate}
  \item\label{C1p}
$n_k/n=\#\set{i:d_i=k}/n\to p_k$ as \ntoo, for every $k\ge0$;
\item \label{C1l}
$\gl\=\sum_k k p_k\in(0,\infty)$;
\item \label{C1d2}
$\sum_i d_i^ 2=O(n)$;
\item \label{C1p1}
$p_1>0$.
\end{romenumerate}
\end{condition}

Let $D_n$ be a random variable defined as the degree of a random
(uniformly chosen) vertex in \gnd{} or \gndx;
thus
\begin{equation}
  \P(D_n=k)=n_k/n.
\end{equation}
Note that $\E D_n=n\qw\sumin d_i=2m/n$.

Further,
let $D$ be a random variable with the distribution $\P(D=k)=p_k$.
Then \ref{C1p} can be written
\begin{equation}
  \label{add}
D_n\dto D.
\end{equation}
In other words, $D$ describes the asymptotic distribution of the degree of a
random vertex in \gnd.
Furthermore, \ref{C1l} is
$\gl=\E D\in(0,\infty)$,
\ref{C1p1} is $\P(D=1)>0$,
and
\ref{C1d2} can be written
\begin{equation}
  \label{aa}
\E D_n^2=O(1).
\end{equation}

\begin{remark}\label{RC1}
  In particular, \eqref{aa} implies that the random variables $D_n$
  are uniformly integrable, and thus \refCN{C1}\ref{C1p}, in the form
  \eqref{add},
 implies
$\E D_n\to\E D$, \ie{}
 \begin{equation}
   \label{mn}
\frac{2m}{n}=n\qw\sumin d_i\to\gl;
 \end{equation}
see \eg{} \cite[Theorems 5.4.2 and 5.5.9]{Gut}.
\end{remark}

Let
\begin{align}
g(x)\= \sumk p_k x^k =\E x^D,
\label{key-funct-2}
\end{align}
the probability generating function of the probability distribution
$(p_k)_{k=0}^\infty$, and define further
\begin{align}
  h(x)&\=
xg'(x)=
\sum_{k=1}^\infty  k p_k x^k,
\label{key-funct}
\\
H(x)&\= \gl x^2 -h(x).
\label{key-funct-1}
\end{align}
Note that 
$h(0)=0$ and $h(1)=\gl$, and thus $H(0)=H(1)=0$.
Note also that
\begin{equation}
H'(1)=2\gl-\sum_k k^2p_k
=  \E (2D-D^2)
= - \E D(D-2).
\end{equation}
See further \refL{LH}.

Our first theorem is essentially the main results of
\citet{MR95,MR98}.

\begin{theorem}
  \label{T1}
Suppose that \refCN{C1} holds and
consider the random  graph \gnd, letting \ntoo.
Let $\cC_1$ and $\cC_2$ be the largest and second largest components of \gnd.
  \begin{romenumerate}
\item\label{T1a}
If $\E D(D-2)=\sum_k k(k-2)p_k>0$, then there is a unique $\xi\in(0,1)$
such that $H(\xi)=0$, or equivalently $g'(\xi)=\gl\xi$,
and
\begin{align*}
  v(\cC_1)/n&\pto 1-g(\xi)>0,
\\
  v_k(\cC_1)/n&\pto p_k(1-\xi^k), \text{ for every } k\ge0,
\\
  e(\cC_1)/n&\pto \tfrac12\gl(1-\xi^2),
\end{align*}
while $ v(\cC_2)/n\pto 0$ and $e(\cC_2)/n\pto 0$.

\item\label{T1b}
If $\E D(D-2)=\sum_k k(k-2)p_k\le0$, then
$ v(\cC_1)/n\pto 0$ and $e(\cC_1)/n\pto 0$.
  \end{romenumerate}

The same results hold for \gndx.
\end{theorem}

In the usual, somewhat informal, language,
the theorem shows that \gnd{} has a giant component if and only if
$\E D(D-2)>0$.

In the critical case, we can be more precise.

\begin{theorem}
  \label{T3}
Suppose that \refCN{C1} holds and that
$\E D(D-2)=\sum_k k(k-2)p_k=0$.
Assume further that
$\gan\=\E D_n(D_n-2)=\sumin d_i(d_i-2)/n>0$
and, moreover, $n\qqq\gan\to\infty$, and that
\begin{equation}\label{C2}
  \sumin d_i^{4+\eta} =O(n)
\end{equation}
for some $\eta>0$.
Let
$\gb\=\E D(D-1)(D-2)$.
Then, $\gb>0$ and
\begin{align*}
  v(\cC_1)&=\frac{2\gl}{\gb} n\gan+\op(n\gan),
\\
  v_k(\cC_1)&=\frac2{\gb}kp_k n\gan+\op(n\gan),
\text{ for every } k\ge0,
\\
  e(\cC_1)&=\frac{2\gl}{\gb} n\gan+\op(n\gan),
\end{align*}
while $ v(\cC_2)=\op(n\gan)$ and $e(\cC_2)/n=\op(n\gan)$.

The same results hold for \gndx.
\end{theorem}

\begin{remark}\label{R3}
  Condition \eqref{C2} may be written $\E D_n^{4+\eta}<\infty$;
it thus implies \eqref{aa} and \refCNN\ref{C1d2};
moreover, it implies that $D_n^2$ and $D_n^3$ are uniformly
integrable.
Hence, using \eqref{add}, \eqref{C2} implies $\E D_n^2\to\E D^2$ and
$\E D_n^3\to\E D^3$.
In particular, the conditions of \refT{T3} imply
\begin{equation}
  \gan\=\E D_n(D_n-2)\to\E D(D-2)=0
\end{equation}
and
\begin{equation}
  \gbn\=\E D_n(D_n-1)(D_n-2)\to\E D(D-1)(D-2)=\gb.
\end{equation}

We do not think that the condition \eqref{C2} is best possible; we
conjecture that, in addition to \refCN{C1}, it is enough to assume
that $D_n^3$ are uniformly
integrable, or, equivalently, that $\E D_n^3\to\E D^3<\infty$.
\end{remark}

\refCN{C1}\ref{C1d2} and \eqref{mn} imply that
\begin{equation}
  \label{simple}
\liminf \P\bigpar{\gndx\text{ is a simple graph}}>0,
\end{equation}
see for instance \citet{b01}, \citet{McKay} or
\citet{McKayWo}
under some extra condition on $\max d_i$
and
\citet{SJ195} for the general case.
Since we obtain \gnd{} by conditioning \gndx{} on being a simple
graph, and all results in Theorems \refand{T1}{T3} are (or can be)
stated in terms of convergence in probability, the results for \gnd{}
follow from the results for \gndx{} by this conditioning.

We will prove Theorems \refand{T1}{T3} for \gndx{}
in Sections \refand{SpfT1}{SpfT3}.
The proofs use the same arguments, but we find it convenient to first
discuss the somewhat simpler case of \refT{T1} in detail and then do
the necessary modifications for \refT{T3}.

\begin{remark}
  \label{R2}
The assumption \refCC{C1d2} is used in our proof mainly for the
reduction to \gndx. In fact, the proof of \refT{T1} for \gndx{} holds
with simple modifications also if \refCC{C1d2} is replaced by the
weaker condition that $D_n$ are uniformly integrable, or equivalently,
see \refR{RC1}, $\E D_n\to\E D$ or \eqref{mn}.
It might also be possible to extend \refT{T1} for \gnd{} too, under
some weaker assumption that \refCC{C1d2}, by combining estimates of
$\P\bigpar{\gndx\text{ is simple}}$ from \eg{} \citet{McKayWo}
with more precise estimates of the error probabilities in
\refS{SpfT1}, but we have not pursued this.
\end{remark}

\begin{remark}
  \refCC{C1p1} excludes the case $p_1=0$; we comment briefly on this
  case here.
Note first that in this case, $\E D(D-2)=\sum_{k=3}^\infty
  k(k-2)p_k\ge0$, with strict inequality as soon as $p_k>0$ for some
  $k\ge3$.

First, if $p_1=0$ and $\E D(D-2)>0$, \ie{} if $p_1=0$ and $\sum_{k\ge
  3}p_k>0$,
it is easily seen (by modifying the proof of \refT{T1} below or by
  adding $\eps n$ verticas of degree 1 and applying \refT{T1})
that all but $\op(n)$ vertices and edges belong to a single giant
  component. Hence, the conclusions of \refT{T1}\ref{T1a} hold with
  $\xi=0$. (In this case, $H(x)>0$ for every $x\in(0,1)$.)

The case $p_1=0$ and $\E D(D-2)=0$, \ie{} $p_k=0$ for all $k\neq 0,2$,
is much more exceptional.
(In this case, $H(x)=0$ for all $x$.)
We give three examples showing that quite different behaviours are possible.
Since isolated vertices do not matter, let us assume $p_0=0$ too and
consider thus the case $p_2=1$.

One example is when all $d_i=2$, so we are studying a random 2-regular
graph.
In this case, the components are cycles. It is well-known, and easy to
see, that (for the multigraph version) the distribution of cycle
lengths is given by the Ewens's sampling formula $\ESF(1/2)$, see
\eg{} \citet{ABT}, and thus $v(\cC_1)/n$ converges in distribution to
a non-degenerate distribution on $[0,1]$ and not to any constant
\cite[Lemma 5.7]{ABT}.
Moreover, the same is true for $v(\cC_2)/n$ (and for $v(\cC_3)/n,
\dots$), so in this case there are several large components.

A second case with $p_2=1$ is obtained by adding a small number of
vertices of degree 1. (More precisely, let $n_1\to\infty$,
$n_1/n\to0$, and $n_2=n-n_1$.) It is then easy to see that
$v(\cC_1)=\opn$.

A third case with $p_2=1$ is obtained by instead adding a small number of
vertices of degree 4 (\ie, $n_4\to\infty$,
$n_4/n\to0$, and $n_2=n-n_4$).
By regarding each vertex of degree 4 as two vertices of degree 2 that
have merged,
it is easy to see that in this case $v(\cC_1)=n-\opn$, so there is a
giant component containing almost everything. (The case $\xi=0$ again.)
\end{remark}

\section{\gnp, \gnm{} and other random graphs}\label{Sother}
The results above can be applied to some other random graphs models
too by conditioning on the vertex degrees;
this works whenever the random graph conditioned on the degree
sequence has a uniform distribution over all possibilities.
Notable examples of such random graphs are \gnp{} and \gnm, and other
examples  are given in
\cite{BrittonDML},
\cite[Section 16.4]{SJ178},
\cite{SJeven}.
If, furthermore, \refCN{C1} and \eqref{C2} hold in probability
(where now $d_i$ are the random vertex degrees),
then Theorems \refand{T1}{T3} hold;
in the latter, we define
$\gan\=\sumin d_i(d_i-2)/n$, which now is random.
(For the proof, it is convenient to use
the Skorohod coupling theorem \cite[Theorem 4.30]{Kallenberg}
and assume that the conditions hold a.s.)

For example, for \gnp{} with $np\to\gl$ or \gnm{} with $2m/n\to\gl$,
where $0<\gl<\infty$, the assumptions hold with $D\sim\Po(\gl)$ and
thus $g(x)=e^{\gl(x-1)}$,
$h(x)=\gl xe^{\gl(x-1)}$,
$H(x)=\gl x\bigpar{x-e^{\gl(x-1)}}$, and we recover the
both the classical threshold $\gl=1$ and
the standard equation $\xi=e^{\gl(\xi-1)}$ for the size of the giant
component when $\gl>1$.

If we consider \gnp{} with $p=(1+\epsn)/n$ where $\epsn\to0$ in
\refT{T3}, we have
$\ga_n/\epsn\pto1$ by the second moment method as soon as
$n\epsn\to\infty$, so we need $n\qqq\epsn\to\infty$ in order to apply
\refT{T3}.
On the other hand, it is well known \cite{b01,giant,JLR}
that if $n\qqq\epsn=O(1)$,
then $v(\cC_1)$ and $v(\cC_2)$ are both of the same order $n^{2/3}$ and
\refT{T3} fails,
which shows that the condition $n\qqq\gan\to\infty$ in \refT{T3}
is best possible.

\section{Finding the largest component}\label{Sfind}

The components of an
arbitrary finite graph or multigraph can be found by the following
standard procedure.
Pick an arbitrary vertex $v$ and determine the component
of $v$ as follows: include all the neighbours of $v$ in an
arbitrary order; then
add in the neighbours of neighbours, and so on, until no more vertices can
be added. The vertices included until this moment form the
component of $v$. If there are still vertices left in the graph, pick any
such vertex $w$, and repeat the above to determine the second
component (the component of vertex $w$). Carry on in this manner until
all the components have been found.

It is clear that we obtain the same result as follows.
Regard each edge as consisting of two {half-edges}, each
half-edge having one endpoint.
We will label the vertices as \emph{sleeping} or \emph{awake} (= used)
and the half-edges as \emph{sleeping}, \emph{active} or \emph{dead};
the sleeping and active half-edges are also called \emph{living}.
We start with all vertices and half-edges sleeping.
Pick a vertex and label its half-edges as active. Then take any active
half-edge, say $x$ and find its partner $y$ in the graph; label these two
half-edges as dead; further, if the endpoint of $y$ is sleeping, label
it as awake and all other half-edges there as active.
Repeat as long as there is any active half-edge.

When there is no active half-edge left, we have obtained the first
component. Then start again with another vertex until all components
are found.

We apply this algorithm to a random multigraph \gndx{} with a given degree
sequence, revealing its edges during the process. We thus observe
initially only the vertex degrees and the half-edges, but not how they
are joined to form edges. Hence, each time we need a partner of an
half-edge, it is uniformly distributed over all other living
half-edges. (The dead half-edges are the ones that already are paired
into edges.)
We make these random choices by giving the half-edges \iid{} random maximal
lifetimes $\tau_x$ with the distribution $\Exp(1)$; in other words, each
half-edge dies spontaneously with rate 1 (unless killed earlier). Each
time we need to find the partner of a half-edge $x$, we then wait
until the next living half-edge $\neq x$ dies and take that one. We
then can formulate an algorithm, constructing $\gndx$ and exploring
its components simultaneously, as follows.
Recall that we start
with all vertices and half-edges sleeping.
\begin{Cenumerate}
  \item\label{Ci}
If there is no active half-edge (as in the beginning), select a
sleeping vertex and declare it awake and all its half-edges active.
For definiteness, we choose the vertex
by choosing a half-edge
uniformly at random among all
sleeping half-edges.
If there is no sleeping half-edge left, the process stops;
the remaining sleeping vertices are all isolated
and we have explored all other components.
  \item\label{Cii}
Pick an active half-edge (which one does not matter) and kill it, \ie,
change its status to dead.
  \item\label{Ciii}
Wait until the next half-edge dies (spontaneously). This half-edge is
joined to the one killed in the previous step \ref{Cii} to form an edge of
the graph. If the vertex it belongs to is sleeping, we change this
vertex to awake and all other half-edges there to active.
Repeat from \ref{Ci}.
\end{Cenumerate}

The components are created between the successive times \ref{Ci} is
performed; the vertices in the component created during one of these
intervals are the vertices that are awakened during the interval.
Note also that a component is completed and \ref{Ci} is performed
exactly when the number of active half-edges is 0 and a half-edge dies
at a vertex where all other half-edges (if any) are dead.

\section{Analysis of the algorithm for \gndx}\label{SpfT1}

Let $S(t)$ and $A(t)$ be the numbers of sleeping and active
half-edges, respectively, at time $t$, and let $L(t)=S(t)+A(t)$ be the
number of living half-edges. As is customary, and for definiteness, we
define these random functions to be right-continuous.

Let us first look at $L(t)$.
We start with $2m$ half-edges, all sleeping and thus living,
but we immediately perform \ref{Ci} and \ref{Cii} and kill one of
them; thus $L(0)=2m-1$. In the sequel, as soon as a living half-edge
dies, we perform \ref{Ciii} and then (instantly) either \ref{Cii}
or both \ref{Ci} and \ref{Cii}.
Since \ref{Ci} does not change the number of living half-edges while
\ref{Cii} and \ref{Ciii} each decrease it by 1, the total result is
that $L(t)$ is decreased by 2 each time one of the living half-edges
dies, except when the last living one dies and the process terminates.

\begin{lemma}
  \label{LL}
As \ntoo,
\begin{equation*}
\suptoo\bigabs{n\qw L(t)-\gl\ee{-2t}} \pto0.
\end{equation*}
\end{lemma}

\begin{proof}
  This (or rather an equivalent statement in a slightly different
  situation) was proved in \cite{SJ184} as a consequence of the
Glivenko--Cantelli theorem
\cite[Proposition 4.24]{Kallenberg}
on convergence of empirical distribution functions. It also follows
  easily from (the proof of) \refL{LL+} below if we replace $\gan$ by 1.
\end{proof}

Next consider the sleeping half-edges. Let $V_k(t)$ be the number of
sleeping \emph{vertices} of degree $k$ at time $t$; thus
\begin{equation*}
  S(t)=\sumki kV_k(t).
\end{equation*}
Note that \ref{Cii} does not affect sleeping half-edges, and that
\ref{Ciii} implies that each sleeping vertex of degree $k$ is eliminated
(\ie, awakened) with intensity $k$, independently of all other
vertices. There are also some sleeping vertices eliminated by
\ref{Ci}.

We first ignore the effect of \ref{Ci} by letting $\tvk(t)$ be the
number of vertices of degree $k$ such that
all its half-edges have maximal lifetimes $\tau_x>t$. (I.e.,
none of its $k$ half-edges would have died
spontaneously up to time $t$, assuming they all escaped \ref{Ci}.)
Let further $\tS(t)\=\sum_k k\tvk(t)$.

\begin{lemma}
  \label{LtS}
As \ntoo,
\begin{equation}\label{lts1}
\suptoo\bigabs{n\qw \tvk(t)-p_k\ee{-kt}} \pto0
\end{equation}
for every $k\ge0$ and
\begin{gather}\label{lts2}
\suptoo\biggabs{n\qw \sumk \tvk(t)-g(\ee{-t})} \pto0,
\\
\suptoo\bigabs{n\qw \tS(t)-h(\ee{-t})} \pto0.
\label{lts3}
\end{gather}
\end{lemma}

\begin{proof}
  The statement \eqref{lts1}, again, follows from the
Glivenko--Cantelli theorem, see \cite{SJ184}, or from the proof of
\refL{LtS+} below. (The case $k=0$ is trivial, with $\tV_0(t)=n_0$ for
all $t$.)

By \refR{RC1}, $D_n$ are uniformly integrable, which means that for
every $\eps>0$ there exists $K<\infty$ such that for all $n$
$\sum_{k>K} kn_k/n = \E(D_n;\,D_n>K)<\eps$.
We may further assume (or deduce by Fatou's inequality)
$\sum_{k>K} p_k<\eps$, and obtain by \eqref{lts1} \whp
\begin{equation*}
  \begin{split}
\suptoo\bigabs{n\qw \tS(t)&-h(\ee{-t})}
=
\suptoo\biggabs{\sumki k\bigpar{n\qw \tvk(t)-p_k\ee{-kt}}}
\\&
\le
\sum_{k=1}^K k\suptoo\bigabs{n\qw \tvk(t)-p_k\ee{-kt}}
 + \sum_{k>K} k\Bigpar{\frac{n_k}{n}+p_k}
\\&
\le\eps+\eps+\eps,
  \end{split}
\end{equation*}
proving \eqref{lts3}. An almost identical argument yields \eqref{lts2}.
\end{proof}

The difference between $S(t)$ and $\tS(t)$ is easily estimated.

\begin{lemma}
  \label{LSS}
If $\dmax\=\max_id_i$ is the maximum degree of \gndx, then
\begin{equation*}
0\le \tS(t)-S(t) <
\sup_{0\le s\le t} \bigpar{\tS(s)-L(s)}+\dmax.
\end{equation*}
\end{lemma}

\begin{proof}
Clearly, $V_k(t)\le\tvk(t)$, and thus $S(t)\le\tS(t)$; furthermore,
$\tS(t)-S(t)$ increases only as a result of \ref{Ci}, which acts to
guarantee that $A(t)=L(t)-S(t)\ge0$.

If \ref{Ci} is performed at time $t$ and a vertex of degree $j>0$ is awakened,
then \ref{Cii} applies instantly and we  have
$A(t)=j-1<\dmax$, and consequently
\begin{equation}\label{ss}
  \tS(t)-S(t)
=
  \tS(t)-L(t)+A(t)
< \tS(t)-L(t)+\dmax.
\end{equation}
Furthermore,
$\tS(t)-S(t)$ is never changed by \ref{Cii} and either unchanged or
decreased by \ref{Ciii}.
Hence, $\tS(t)-S(t)$ does not increase until the next time \ref{Ci} is
performed. Consequently, for any time $t$, if $s$ was the last time
before (or equal to) $t$ that \ref{Ci} was performed, then
$\tS(t)-S(t)\le   \tS(s)-S(s)$, and the result follows by \eqref{ss}.
\end{proof}

Let
\begin{equation}\label{ta}
  \tA(t)\=L(t)-\tS(t)
=A(t)-\bigpar{\tS(t)-S(t)}.
\end{equation}
Then, by Lemmas \refand{LL}{LtS} and \eqref{key-funct-1},
\begin{equation}\label{tah}
\suptoo\bigabs{n\qw \tA(t)-H(\ee{-t})} \pto0.
\end{equation}
\refL{LSS} can be written
\begin{equation}
  \label{tssa}
0\le \tS(t)-S(t)<- \inf_{s\le t}\tA(s)+\dmax.
\end{equation}

\begin{remark}
  By \eqref{ta} and \eqref{tssa}, we obtain further the relation
\begin{equation*}
  \tA(t) \le A(t)
< \tA(t) - \inf_{s\le t}\tA(s)+\dmax
\end{equation*}
which, perhaps, illuminates the relation between $A(t)$ and $\tA(t)$.
\end{remark}

\begin{lemma}
  \label{LH}
Suppose that \refCN{C1} holds and let $H(x)$ be given by \eqref{key-funct-1}.
  \begin{romenumerate}
\item\label{LHa}
If\/ $\E D(D-2)=\sum_k k(k-2)p_k>0$, then there is a unique $\xi\in(0,1)$
such that $H(\xi)=0$; 
moreover,
$H(x)<0$ for $x\in(0,\xi)$ and
$H(x)>0$ for $x\in(\xi,1)$.
\item\label{LHb}
If\/ $\E D(D-2)=\sum_k k(k-2)p_k\le0$, then
$H(x)<0$ for $x\in(0,1)$.
  \end{romenumerate}
\end{lemma}

\begin{proof}
As remarked earlier,
$H(0)=H(1)=0$ and
$H'(1)= - \E D(D-2)$.
Furthermore, if we define $\phi(x)\=H(x)/x$,
then  $\phi(x)=\gl x - \sum_k kp_k x^{k-1}$ is a concave function on
  $(0,1]$, and it is strictly concave unless
$p_k=0$ for all $k\ge3$,
in which case $H'(1)=-\E D(D-2)=p_1>0$.

In case \ref{LHb}, we thus have $\phi$ concave and
$\phi'(1)=H'(1)-H(1)\ge0$, with either the concavity or the inequality strict,
and thus $\phi'(x)>0$ for all $x\in(0,1)$,
whence $\phi(x)<\phi(1)=0$ for $x\in(0,1)$.

In case \ref{LHa}, $H'(1)<0$, and thus
$H(x)>0$ for $x$ close to 1.
Further, $H'(0)=-h'(0)=-p_1<0$,
and thus
$H(x)<0$ for $x$ close to 0.
Hence there is at least one $\xi\in(0,1)$
with $H(\xi)=0$, and since $H(x)/x$ is strictly concave and also
$H(1)=0$, there is at most one such $\xi$ and the result follows.
\end{proof}

\begin{proof}[Proof of \refT{T1}\ref{T1a}]
Let $\xi$ be the zero of $H$ given by \refL{LH}\ref{LHa} and let
$\tau\=-\ln\xi$. Then, by \refL{LH},
$H(\eet)>0$ for $0<t<\tau$, and thus
$\inf_{t\le\tau} H(\eet)=0$.
Consequently, \eqref{tah} implies
\begin{equation}\label{tc}
{n\qw\inf_{t\le\tau} \tA(t)}
=
{\inf_{t\le\tau}n\qw \tA(t)-
\inf_{t\le\tau} H(\eet)} \pto0.
\end{equation}

Further, by \refCN{C1}\ref{C1d2}, $\dmax=O(n\qq)$, and thus $n\qw\dmax\to0$.
Consequently, \eqref{tssa} and \eqref{tc} yield
\begin{equation}
  \label{tc1}
\sup_{t\le\tau}n\qw\bigabs{ A(t)-\tA(t)}
=
\sup_{t\le\tau}n\qw\bigabs{ \tS(t)-S(t)}
\pto0
\end{equation}
and thus, by \eqref{tah},
\begin{equation}\label{tc2}
  \sup_{t\le\tau}\bigabs{n\qw A(t)-H(\ee{-t})} \pto0.
\end{equation}

Let $0<\eps<\tau/2$. Since $H(\eet)>0$ on the compact interval
$[\eps,\tau-\eps]$, \eqref{tc2} implies that \whp{} $A(t)$ remains
positive on $[\eps,\tau-\eps]$, and thus no new component is started
during this interval.

On the other hand, again by \refL{LH}\ref{LHa},
$H(\ee{-\tau-\eps})<0$ and \eqref{tah} implies
$n\qw\tA(\tau+\eps)\pto H(\ee{-\tau-\eps})$,
while $A(\tau+\eps)\ge0$.
Thus, with $\gd\=|H(\ee{-\tau-\eps})|/2>0$, \whp
\begin{equation}  \label{sste}
  \tS(\tau+\eps)- S(\tau+\eps)
=
 A(\tau+\eps) -  \tA(\tau+\eps)
\ge -\tA(\tau+\eps)
>n\gd,
\end{equation}
while \eqref{tc1} yields $\tS(\tau)-S(\tau)<n\gd$ \whp.
Consequently, \whp{}
$  \tS(\tau+\eps)- S(\tau+\eps)>  \tS(\tau)- S(\tau)$, so
\ref{Ci} is performed between $\tau$ and $\tau+\eps$.

Let $T_1$ be the last time \ref{Ci} was performed before $\tau/2$ and
let $T_2$ be the next time it is performed. We have shown that for any
$\eps>0$, \whp{} $0\le T_1\le \eps$ and $\tau-\eps\le T_2\le \tau+\eps$;
in other words, $T_1\pto0$ and $T_2\pto \tau$.

We state the next step as a lemma that we will reuse.

\begin{lemma}
  \label{LC}
Let\/ $T_1^*$ and $T_2^*$ be two (random) times when \ref{Ci} are performed,
with $T^*_1\le T^*_2$, and assume that $T^*_1\pto t_1$ and $T^*_2\pto
t_2$ where
$0\le t_1\le t_2\le \tau$.
If\/ $C^*$ is the union of all components explored between $T^*_1$ and
$T^*_2$, then
\begin{align}
  v_k(C^*)/n &\pto p_k\bigpar{\eez{kt_1}-\eez{kt_2}},
\qquad k\ge0,
\label{lcvk} \\
  v(C^*)/n &\pto g(\eez{t_1})-g(\eez{t_2}),
\label{lcv} \\
  e(C^*)/n &\pto \thalf h(\eez{t_1})-\thalf h(\eez{t_2}).
\label{lce}
\end{align}
In particular, if $t_1=t_2$, then $v(C^*)/n\pto0$ and $e(C^*)/n\pto0$.
\end{lemma}

\begin{proof}
$C^*$ contains all vertices awakened in the interval $[T^*_1,T^*_2)$
and no others, and thus
\begin{equation}\label{td}
v_k(C^*)=V_k(T^*_1-)-V_k(T^*_2-),
\qquad k\ge1.
\end{equation}

Since $T^*_2\pto t_2\le\tau$ and $H$ is continuous,
$
\inf_{t\le T^*_2} H(t) \pto
\inf_{t\le t_2} H(t)=0$,
and \eqref{tah} and \eqref{tssa} imply, in analogy with \eqref{tc} and
\eqref{tc1},
$n\qw \inf_{t\le T^*_2} \tA(t)\pto0$ and
\begin{equation}
  \label{sst2}
\sup_{t\le T^*_2} n\qw|\tS(t)-S(t)|\pto0.
\end{equation}

Since $\tV_j(t)\ge V_j(t)$ for every $j$ and $t\ge0$,
\begin{equation}
  \label{td1}
\tvk(t)-V_k(t)\le k\qw\sumji j\bigpar{\tV_j(t)-V_j(t)}
=k\qw\bigpar{\tS(t)-S(t)},
\qquad k\ge1.
\end{equation}
Hence \eqref{sst2} implies, for every $k\ge1$,
$\sup_{t\le T^*_2}|\tvk(t)-V_k(t)|=\op(n)$.
This is further trivially true for $k=0$ too.
Consequently, using  \refL{LtS}, for $j=1,2$,
\begin{equation}
  \label{stj}
V_k(T^*_j-)
=
\tV_k(T^*_j-)+\op(n)
=np_k\eez{kT^*_j}+\op(n)
=np_k\eez{kt_j}+\opn,
\end{equation}
and \eqref{lcvk} follows by \eqref{td}.
Similarly, using $\sumk(\tvk(t)-V_k(t))\le\tS(t)-S(t)$,
\begin{equation*}
  \begin{split}
  v(C^*)
&=\sumki\bigpar{V_k(T^*_1-)-V_k(T^*_2-)}
=\sumki\bigpar{\tvk(T^*_1-)-\tvk(T^*_2-)}+\opn
\\&
=ng(\eez{T^*_1})-ng(\eez{T^*_2})+\opn
  \end{split}
\end{equation*}
and
\begin{equation*}
  \begin{split}
  2e(C^*)
&=\sumki k\bigpar{V_k(T^*_1-)-V_k(T^*_2-)}
=\sumki k\bigpar{\tvk(T^*_1-)-\tvk(T^*_2-)}+\opn
\\&
=nh(\eez{T^*_1})-nh(\eez{T^*_2})+\opn,
  \end{split}
\end{equation*}
and \eqref{lcv} and \eqref{lce} follow.
\end{proof}

Let $\pC$ be the component created at $T_1$ and explored until
$T_2$.
By \refL{LC}, with $t_1=0$ and $t_2=\tau$,
\begin{align}
v_k(\pC)/n&\pto p_k(1-\eez{k\tau}),
\label{vc'k}
\\
v(\pC)/n&\pto g(1)-g(\eez{\tau})=1-g(\xi),
\label{vc'v}
\\
e(\pC)/n
&\pto
\tfrac12\bigpar{h(1)-h(\eez\tau)}
=\tfrac12\bigpar{h(1)-h(\xi)}
=\frac\gl2(1-\xi^2),
\label{vc'e}
\end{align}
using \eqref{key-funct-1} and $H(1)=H(\xi)=0$.

We have found one large component $\pC$ with the claimed numbers of
vertices and edges. It remains to show that there is \whp{} no other
large component.
Therefore,
let $T_3$ be the first time after $T_2$ that \ref{Ci} is
performed. Since $\tS(t)-S(t)$ increases by at most $\dmax=\opn$
each time \ref{Ci} is performed, we obtain from \eqref{sst2} that
\begin{equation*}
\sup_{t\le T_3}\bigpar{\tS(t)-S(t)}
\le
\sup_{t\le T_2}\bigpar{\tS(t)-S(t)}+\dmax=\opn.
\end{equation*}
Comparing this to \eqref{sste} we see that for every $\eps>0$,
\whp{} $\tau+\eps>T_3$. Since also $T_3>T_2\pto\tau$, it follows that
$T_3\pto\tau$. If $\ppC$ is the component created between $T_2$ and
$T_3$, then
\refL{LC} applied to $T_2$ and $T_3$ yields
$v(\ppC)/n\pto0$
and
$e(\ppC)/n\pto0$.

Next, let $\eta>0$.
Applying
\refL{LC} to $T_0\=0$ and $T_1$, we see that
the total number of vertices and edges
in all components found before $\pC$, \ie, before $T_1$, is
$\opn$, because $T_1\pto0$.
Hence, recalling $m=\Theta(n)$ by \eqref{mn},
\begin{equation}\label{cc1}
 \P(\text{a component $\cC$ with $e(\cC)\ge\eta m$ is found before $\pC$})
\to0.
\end{equation}
On the other hand,
conditioning on the final graph \gndx{} that is constructed by the
algorithm,
if there exists a component $\cC\neq\pC$ in \gndx{}
with at least $\eta m$ edges that has not been found before $\pC$, then with
probability at least $\eta$, the vertex chosen at random by \ref{Ci}
at $T_2$
starting the component $\ppC$ belongs to $\cC$, and thus
$\cC=\ppC$. Consequently,
\begin{multline}\label{cc2}
 \P(\text{a component $\cC$ with $e(\cC)\ge\eta m$ is found after $\pC$})
\\
\le
\eta\qw
\P(e(\ppC)\ge\eta m)
    \to0.
\end{multline}
Combining \eqref{cc1} and \eqref{cc2}, we see that \whp{} there is no
component except $\pC$ with at least $\eta m$ edges.
Taking $\eta$
small, this and \eqref{vc'e} show that \whp{} $\pC=\cC_1$, the largest
component, and
further $e(\cC_2)<\eta m$.
Consequently, the results for $\cC_1$ follow from
\eqref{vc'k}--\eqref{vc'e}. We have further shown
$e(\cC_2)/m\pto0$, which implies
$e(\cC_2)/n\pto0$ and $v(\cC_2)/n\pto0$ because $m=O(n)$ and
$v(\cC_2)\le e(\cC_2)+1$.
\end{proof}

\begin{proof}[Proof of \refT{T1}\ref{T1b}]
This is very similar to the last step in the proof for \ref{T1a}.
Let $T_1=0$ and let $T_2$ be the next time \ref{Ci} is performed.
Then
\begin{equation}
  \label{sofie}
\sup_{t\le T_2}\bigabs{A(t)-\tA(t)}
=
\sup_{t\le T_2}\bigabs{\tS(t)-S(t)}
\le 2\dmax=o(n).
\end{equation}
For every $\eps>0$, we have by \eqref{tah} and \refL{LH}\ref{LHb}
$n\qw \tA(\eps)\pto H(\eez{\eps})<0$, while $A(\eps)\ge0$, and it
follows from \eqref{sofie} that \whp{} $T_2<\eps$. Hence, $T_2\pto0$.
We apply \refL{LC} (which holds in this case too, with $\tau=0$) and
find that if $\tC$ is the first component, then $e(\tC)/n\pto0$.

Let $\eps>0$. If $e(\cC_1)\ge\eps m$, then the probability that the
first half-edge chosen by \ref{Ci} belongs to $\cC_1$, and thus
$\tC=\cC_1$, is $2e(\cC_1)/(2m)\ge\eps$, and hence, using
$m=\Theta(n)$ by \eqref{mn},
\begin{equation*}
  \P\bigpar{e(\cC_1)\ge\eps m}
\le \eps\qw   \P\bigpar{e(\tC)\ge\eps m}\to0.
\end{equation*}
The results follows since
$m=O(n)$ and $v(\cC_1)\le e(\cC_1)+1$.
\end{proof}

\section{Proof of \refT{T3} for \gndx}\label{SpfT3}

We assume in this section that the assumptions of \refT{T3} hold.
Note first that $\gb\=\E D(D-1)(D-2)\ge0$ with strict inequality
unless $\P(D\le2)=1$, but in the latter case $\ga=\E D(D-2)=-p_1<0$,
which is ruled out by the assumptions.

Define, in analogy with \eqref{key-funct-2}--\eqref{key-funct-1},
\begin{align*}
g_n(x)&\= \E x^{D_n}=\sumk \nkn x^k,
\\
  h_n(x)&\=
xg_n'(x)=
\sum_{k=1}^\infty k \nkn x^k,
\\
H_n(x)&\= \frac{2m(n)}{n}x^2 - h_n(x)
=
\sum_{k=1}^\infty  k \frac{n_k}{n}(x^2- x^k)
=
\E D_n(x^2-x^{D_n}).
\end{align*}

We begin with a general estimate for death processes and use it to
prove estimates improving Lemmas \refand{LL}{LtS}.

\begin{lemma}
\label{L2}
Let $\gamma>0$ and $d>0$ be fixed.
Let $N^{(x)}(t)$ be a Markov process such that $N^{(x)}(0) =x$ a.s. and
transitions are made according to the following rule: whenever
in state $y>0$, the process jumps to $y-d$ with intensity $\gamma y$;
in other words, the waiting time until the next event is $\Exp(1/\gamma y)$ and
each jump is of size $d$ downwards.
Then, for every $t_0\ge0$,
\begin{equation}\label{magnus}
\E\suptto \bigl| N^{(x)}(t)-e^{-\gamma  dt}x\bigr|^2
\le 8d\bigpar{\ee{\gam d t_0}-1}x+8d^2.
\end{equation}
If $x/d$ is an integer, we also have the better estimate
\begin{equation}\label{erika}
\E\suptto \bigl| N^{(x)}(t)-e^{-\gamma  dt}x\bigr|^2
\le 4d\bigpar{\ee{\gam d t_0}-1}x.
\end{equation}
\end{lemma}

\begin{proof}
  First assume that $d=1$ and that $x$ is an integer. In this case,
  the process is a standard pure death process taking the values
$x, x-1,\dots, 0$, describing the number of particles alive when the
  particles die independently with rate $\gam$.
As is well-known, and easily seen by regarding $\nx(t)$ as the sum of
  $x$ independent copies of the process $N^{(1)}(t)$, the process
$\ee{\gam t}\nx(t)$, $t\ge0$, is a martingale. Furthermore, for every
  $t\ge0$, $\nx(t)\sim\Bi(x,\eez{\gam t})$.
Hence,
by Doob's inequality,
\begin{equation}\label{e1}
  \begin{split}
\E \sup_{t\le t_0} \lrabs{\nx(t)-x\eez{\gam t}}^2
&\le \E \sup_{t\le t_0} \lrabs{\ee{\gam t}\nx(t)-x}^2
\le 4 \E \lrabs{e^{\gam t_0}\nx(t_0)-x}^2
\\
&=4\ee{2\gam t_0}\Var\nx(t_0)
=4(e^{\gam t_0}-1)x.
  \end{split}
\raisetag{12pt}
\end{equation}

Next, still assume $d=1$ but let $x\ge0$ be arbitrary.
We can couple the two processes $\nx(t)$ and $\nxf(t)$ with
different initial values such that whenever the smaller one jumps (by
$-1$), so does the other. This coupling keeps
$|\nx(t)-\nxf(t)|<1$ for all $t\ge0$, and thus,
\begin{equation*}
\sup_{t\le t_0} \lrabs{\nx(t)-x\eez{\gam t}}
\le
\sup_{t\le t_0} \lrabs{\nxf(t)-\floor{x}\eez{\gam t}}+2
\end{equation*}
and hence by \eqref{e1}
\begin{equation}\label{e2}
\E\sup_{t\le t_0} \lrabs{\nx(t)-x\eez{\gam t}}^2
\le
8\bigpar{\ee{\gam t_0}-1}x+8.
\end{equation}

Finally, for a general $d>0$ we observe that $\nx(t)/d$ is a process
of the same type with the parameters $(\gam,d,x)$ replaced by
$(\gam d,1,x/d)$, and the general result follows from \eqref{e2} and
\eqref{e1}.
\end{proof}

\begin{lemma}
  \label{LL+}
For every fixed $t_0>0$,
as \ntoo,
\begin{equation*}
\suptanto\bigabs{L(t)-2 m(n)\ee{-2 t}} =\Op\bigpar{n\qq\gan\qq+1}.
\end{equation*}
\end{lemma}
\begin{proof}
$L(t)$ is a death process as in \refL{L2}, with $\gam=1$, $d=2$ and
  $x=L(0)=2m(n)-1$.
Hence, by \refL{L2} applied to $\gan t_0$, observing that
$\gan t_0=O(\gan)=O(1)$ and $m(n)=O(n)$,
\begin{equation*}
\E\suptanto\bigabs{L(t)-2 m(n)\ee{-2t}}^2
=O\bigpar{(e^{2\gan t_0}-1)m(n)+1}
=O\bigpar{\gan n+1}.
\qedhere
\end{equation*}
\end{proof}

\begin{lemma}
  \label{LtS+}
For every fixed $t_0\ge0$
\begin{equation*}
\suptanto\bigabs{\tvk(t)-n_k\ee{-kt}}= \Op\bigpar{n\qq\gan\qq}
\end{equation*}
for every $k\ge0$ and
\begin{gather*}
\suptanto\bigabs{ \sumk \tvk(t)-ng_n(\ee{-t})}
= \Op\bigpar{n\qq\gan\qq+\ngaaa},
\\
\suptanto\bigabs{ \tS(t)-nh_n(\ee{-t})}
= \Op\bigpar{n\qq\gan\qq+\ngaaa}.
\end{gather*}
\end{lemma}

\begin{proof}
  $\tvk(t)$ is a death process as in \refL{L2} with $\gam=k$, $d=1$
  and $x=n_k$. Consequently, by \eqref{erika}
(for $k\ge1$; the case $k=0$ is trivial),
\begin{equation}\label{e5}
\E\suptanto\bigabs{\tvk(t)-n_k\ee{-kt}}^2
\le 4\bigpar{\ee{k\gan t_0}-1}n_k
\le 4k\gan t_0\ee{k\gan t_0}n_k.
\end{equation}
The estimate for fixed $k$ follows immediately.

To treat $\tS(t)=\sumki k\tvk(t)$, we use \eqref{e5} for $k\le\gan\qw$
and obtain 
\begin{equation*}
\E\suptanto\bigabs{\tvk(t)-n_k\ee{-kt}}
\le \bigpar{4 t_0\ee{t_0}k\gan n_k}\qq,
\qquad
k\le\gan\qw.
\end{equation*}
For $k>\gan\qw$ we use the trivial estimate
$\sup_t|\tvk(t)-n_k\eez{kt}|\le n_k$. Summing over $k$ and using
the \CS{} inequality and \eqref{C2} we find, for some $C$ depending on $t_0$,
\begin{multline*}
\E\suptanto\lrabs{\tS(t)-\sumki kn_k\eez{kt}}
\le
\E\sumki k\suptanto\lrabs{\tvk(t)-n_k\eez{kt}}
\\[-6pt]
\begin{aligned}
\\&
\le
C\sum_{k\le\gan\qw}k(k\gan n_k)\qq +\sum_{k>\gan\qw}k n_k
\\&
\le
C\gan\qq\Bigpar{\sumki k^{4+\eta} n_k}\qq \Bigpar{\sumki k^{-1-\eta}}\qq
+\gan^3\sum_{k>\gan\qw}k^4 n_k
\\&
=O\bigpar{n\qq\gan\qq+\ngaaa}.
 \end{aligned}
\end{multline*}
The estimate for
$\sumk \tvk(t)$ is proved the same way.
\end{proof}

Lemmas \ref{LL+} and \ref{LtS+} imply, \cf{} \eqref{ta} and
\eqref{tah}, for every $t_0>0$,
\begin{equation}\label{e7}
  \begin{split}
\gan\qww\sup_{t\le t_0}\bigabs{n\qw\tA(\gan &t)-H_n(\eez{\gan t})}
\\&
=\gan\qww n\qw\sup_{t\le \gan t_0}\bigabs{L(t)-\tS(t)-nH_n(\eez{t})}
\\&
=\Op\bigpar{n\qqw\gan\qqcw+n\qw\gan\qww+\gan}
=\op(1),
  \end{split}
\end{equation}
recalling $\gan\to0$ by \refR{R3} and $n\gan^3\to\infty$.

Let
\begin{equation*}
\hhn(t)\= H_n(\eet)=\E D_n\bigpar{\eez{2t}-\eez{t D_n}}.
\end{equation*}
Then $\hhn(0)=0$, $\hhn'(0)=\E D_n(D_n-2)=\gan$,
$\hhn''(0)=\E D_n(4-D_n^2)=-\E D_n(D_n+2)(D_n-2)$, and for all $t\ge0$,
\begin{equation*}
  \lrabs{\hhn'''(t)}
=\lrabs{\E\bigpar{D_n\bigpar{8\eez{2t}-D_n^3\eez{tD_n}}}}
\le
\E\bigpar{8D_n+D_n^4}
=O(1).
\end{equation*}
Moreover, using \refR{R3},
\begin{equation*}
  \begin{split}
  \hhn''(0)
&=\,-\E D_n(D_n+2)(D_n-2)
\\&
\to-\E D(D+2)(D-2)
\\&
=\,-\E D(D-1)(D-2)-3\E D(D-2)
=-\gb.
  \end{split}
\end{equation*}
Hence a Taylor expansion yields, for $t\ge0$,
\begin{equation*}
  \begin{split}
\hhn(\gan t)
=\gan^2t+\thalf\hhn''(0)(\gan t)^2+O\bigpar{(\gan t)^3}
=\gan^2\lrpar{t-\thalf\gb t^2+o(t^2+ t^3)}.
  \end{split}
\end{equation*}
Consequently, \eqref{e7} yields, for every fixed $t_0>0$,
\begin{equation}
  \label{e8}
\suptto\lrabs{\gan\qww n\qw \tA(\gan t)-(t-\thalf\gb t^2)}
=\op(1).
\end{equation}

We now proceed as in the proof of \refT{T1},
using $\tH(t)\=t-\half\gb t^2$ instead of $H(\eet)$.
We note that $\tH(t)>0$ for $0<t<2/\gb$ and
$\tH(t)<0$ for $t>2/\gb$; thus we now define $\tau=2/\gb$.
We obtain from \eqref{e8}, for any random $T\pto\tau$,
\begin{equation*}
  \gan\qww n\qw \inf_{t\le T}\tA(\gan t)\pto0
\end{equation*}
and, using \eqref{tssa}, since by \eqref{C2}
$\dmax=o(n^{1/3})=o(n\gan^2)$,
\begin{equation}\label{e11}
 \gan\qww n\qw \sup_{t\le T}|A(\gan t)-\tA(\gan t)|
=
 \gan\qww n\qw \sup_{t\le T}|\tS(\gan t)-S(\gan t)|\pto0
\end{equation}
and thus, by \eqref{e8} again,
\begin{equation*}
 \sup_{t\le T}| \gan\qww n\qw A(\gan t)-\tH(t)|\pto0.
\end{equation*}

Taking $T=\tau$,
it follows as in \refS{SpfT1} that \whp{} there is a component $\pC$
explored between two
random times $T_1$ and $T_2$ with $T_1/\gan\pto0$ and
$T_2/\gan\pto\tau=2/\gb$.
We have the following analogue of \refL{LC}.

\begin{lemma}
  \label{LC+}
Let\/ $T_1$ and $T_2$ be two (random) times when \ref{Ci} are performed,
with $T_1\le T_2$, and assume that $T_1/\gan\pto t_1$ and
$T_2/\gan\pto t_2$ where
$0\le t_1\le t_2\le \tau=2/\gb$.
If\/ $\tC$ is the union of all components explored between $T_1$ and
$T_2$, then
\begin{align*}
  v_k(\tC)
&=n\gan k p_k(t_2-t_1)+\op(n\gan),
\qquad k\ge0,
\\
  v(\tC)
&=n\gan\gl(t_2-t_1)+\op(n\gan),
\\
  e(\tC)
&=n\gan\gl(t_2-t_1)+\op(n\gan).
\end{align*}
In particular, if $t_1=t_2$, then $v(\tC)=\op(n\gan)$ and $e(\tC)=\op(n\gan)$.
\end{lemma}

\begin{proof}
$\tC$ consists of the vertices awakened in the interval $[T_1,T_2)$,
    and thus,
using \eqref{e11} and \refL{LtS+},
\begin{align*}
v_k(\tC)
&=V_k(T_1-)-V_k(T_2-)
=\tvk(T_1-)-\tvk(T_2-)+\op(n\gan^2)
\\&
=n_k\bigpar{\eez{kT_1}-\eez{kT_2}}+\op(n\gan^2)
\\&
=n_k\bigpar{kT_2-kT_1+\Op(\gan^2)}+\op(n\gan^2)
\\&
=n_k k\gan(t_2-t_1)+\op(n\gan)
\\&
= kp_k n\gan(t_2-t_1)+\op(n\gan).
\end{align*}
Further, since
$g_n'(1)=\E D_n\to\E D=\gl$ and
$g_n''(x)=O(1)$ for $0<x<1$,
a Taylor expansion yields, for $j=1,2$,
\begin{equation*}
g_n\bigpar{\eez{T_j}}
=g_n(1)-g_n'(1) T_j +O(T_j^2)
=1-\gl \gan t_j + \op(\gan).
\end{equation*}
Similarly, since
$h_n'(1)=\E D_n^2=\gan+2\E D_n\to2\gl$,
and $h_n''(x)=O(1)$ for $0<x<1$,
\begin{equation*}
h_n\bigpar{\eez{T_j}}
=h_n(1)-h_n'(1) T_j +O(T_j^2)
=\E D_n-2\gl \gan t_j + \op(\gan).
\end{equation*}
It now follows from \eqref{e11} and \refL{LtS+} that
\begin{align*}
v(\tC)&=ng_n\xpar{\eez{T_1}}-ng_n\xpar{\eez{T_2}}+\op(n\gan^2)
=n\gan\gl(t_2-t_1)+\op(n\gan),
\\
2e(\tC)&=nh_n\xpar{\eez{T_1}}-nh_n\xpar{\eez{T_2}}+\op(n\gan^2)
=2n\gan\gl(t_2-t_1)+\op(n\gan).
\end{align*}
\end{proof}

In particular, for the component $\pC$ found above, with $t_1=0$ and
$t_2=\tau$,
\begin{align}
v_k(\pC)&= kp_k\tau n\gan+\op(n\gan),
\\
v(\pC)&=\gl\tau n\gan+\op(n\gan), 
\\
e(\pC)&=\gl\tau n\gan+\op(n\gan). \label{ecq}
\end{align}

Since $\tau=2/\gb$, these are the estimates we claim for $\cC_1$, and
it remains only to show that \whp{} all other components are much
smaller than $\pC$.

Fix $\eps>0$ with $\eps<\tau$, and say that a component of \gndx{}
is \emph{large} if it has at least $\eps m\gan$ edges ($2\eps m\gan$
half-edges). Since, by \eqref{ecq} and \eqref{mn}, $e(\pC)/(m\gan)\pto2\tau$,
\whp{} $\pC$ is large, and further
$(2\tau-\eps)m\gan<e(\pC)<(2\tau+\eps)m\gan$.
Let $\eee$ be the event that
$e(\pC)<(2\tau+\eps)m\gan$ and that the total number of edges in large
components is at least
$(2\tau+2\eps)m\gan$.

It follows by \refL{LC+} applied to $T_0=0$ and $T_1$ that the total
number of vertices or edges in components found before $\pC$ is
$\op(n\gan)$.
Thus there exists a sequence $\gan'$ of constants such that
$\gan'=o(\gan)$ and \whp{} at most $n\gan'$ vertices are found before
$T_1$, when the first large component is found.

Let us now condition on the final
graph obtained through our component-finding algorithm. Given \gndx,
the components appear in our process in the size-biased
order (with respect to the number of edges) obtained by
picking half-edges uniformly at random (with replacement, for
simplicity) and taking the corresponding components, ignoring every
component that already has been taken. We have seen that \whp{} this
finds component containing at most $n\gan'$ vertices before a
half-edge in a large component is picked. Therefore, starting again
at $T_2$, \whp{} we find at most $n\gan'$ vertices in new components
before a half-edge is chosen in some large component; this half-edge
may belong to $\pC$, but if $\eee$ holds, then with probability at
least $\eps_1\=1-(2\tau+\eps)/(2\tau+2\eps)$ it does not, and
therefore it belongs to a new large component. Consequently, with
probability at least $\eps_1\P(\eee)+o(1)$, the algorithm in
\refS{Sfind} finds a second large component at a time $T_3$, and
less than $n\gan'$ vertices between $T_2$ and $T_3$. In this case,
let $T_4$ be the time this second large component is completed. (If
no such second large component is found, let for definiteness
$T_3=T_4=T_2$.)

Note that
$0\le \tV_1(t)-V_1(t)\le \tS(t)-S(t)$ for all $t$.
Hence, using
\refL{LtS+} and \eqref{e11} with $T=T_2/\gan$,
the number of vertices of degree 1 found between $T_2$ and $T_3$ is
\begin{equation*}
  \begin{split}
V_1(T_2-)-V_1(T_3-)
&\ge
\tV_1(T_2-)-(\tS(T_2-)-S(T_2-))-\tV_1(T_3-)
\\&
=n_1\eez{T_2}-n_1\eez{T_3}+\op(n\gan^2).
  \end{split}
\end{equation*}
Since this is at most $n\gan'=o(n\gan)$, and $n_1/n\to p_1>0$, it
follows that
$\eez{T_2}-\eez{T_3}=\op(\gan)$,
and thus
$T_3=T_2+\op(\gan)=\tau\gan+\op(\gan)$.
Hence, \eqref{e11} applies to $T=T_3/\gan$, and since no \ref{Ci} is
performed between $T_3$ and $T_4$,
\begin{equation}
  \label{f1}
\sup_{t\le T_4}\bigabs{\tS(t)-S(t)}
\le
\sup_{t\le T_3}\bigabs{\tS(t)-S(t)}+\dmax
=\op(n\gan^2).
\end{equation}

Let $t_0>\tau$; thus $\tH(t_0)=t_0-\half\gb t_0^2<0$ and \eqref{e8}
yields, with $\gd=|\tH(t_0)|/2>0$,
\whp{} $\tA(\gan t_0) \le -n\gan^2\gd$ and thus
\begin{equation*}
  \tS(\gan t_0)-S(\gan t_0)
=A(\gan t_0)-\tA(\gan t_0) \ge n\gan^2\gd.
\end{equation*}
Hence \eqref{f1} shows that \whp{} $T_4<\gan t_0$. Since $t_0>\tau$ is
arbitrary, and further $T_2\le T_3\le T_4$ and $T_2/\gan\pto\tau$, it
follows that $T_3/\gan\pto\tau$ and $T_4/\gan\pto\tau$.

Finally, by \refL{LC+} again, this time applied to $T_3$ and $T_4$, the
number of edges found between $T_3$ and $T_4$ is
$\op(n\gan)=\op(m\gan)$. Hence,
\whp{} there is no large component found there, although the
construction gave a large component with probability at least
$\eps_1\P(\eee)+o(1)$.
Consequently,
$\eps_1\P(\eee)=o(1)$ and
$\P(\eee)=o(1)$.

Recalling the definition of $\eee$, we see that \whp{} the total
number of edges in large components is at most $(2\tau+2\eps)m\gan$;
since \whp{} at least $(2\tau-\eps)m\gan$ of these belong to $\pC$,
there are at most $3\eps m\gan$ edges, and therefore
at most $3\eps m\gan+1$ vertices, in any other component.

Choosing $\eps$ small enough, this shows that \whp{} $\cC_1=\pC$, and
further
$v(\cC_2)\le e(\cC_2)+1\le 3\eps m\gan+1<3\gl\eps n\gan$.
\qed

\section{Conceivable extensions}

It seems to be possible to
obtain quantitative versions of our results, such as
a central limit theorem for the size of the giant component,
as we did for the $k$-core in \cite{SJ196}.
(See \citet{Pittel} and \citet{BBV} for \gnp{} and \gnm, and
\cite{LL06} for the random cluster model.)
Similarly, it should be possible to obtain large deviation estimates.

Further, in the transition window, where $\gan=O(n\qqq)$,
an appropriate scaling seems to lead to convergence to Gaussian processes
resembling the one studied by \citet{Aldous}, and it seems likely that
similar results on the distribution of the sizes of the largest
components could be obtained.

We have not attempted above to give more precise bounds on the size
of the second component $\cC_2$, and we leave it as an open problem
to see whether our methods can lead to new insights for this
problem. It appears that direct analysis of the Markov process
$(A(t),V_0(t),V_1(t), \ldots)$ can show that the largest component
has size $O(\log n)$ in the subcritical phase, and that so does the
second largest component in the supercritical case, but we have not
pursued this.

Finally, it seems possible to adapt the methods of this paper to
random hypergraphs and obtain results similar to those
in~\citet{BCK07}, but we leave this to the reader.

\begin{acks}
This research was partly done during a visit by SJ to the University of
Cambridge, partly funded by Trinity College. MJL was partly supported
by the Nuffield Foundation and by a Research Fellowship from the
Humboldt Foundation.
\end{acks}

\newcommand\AAP{\emph{Adv. Appl. Probab.} }
\newcommand\JAP{\emph{J. Appl. Probab.} }
\newcommand\JAMS{\emph{J. \AMS} }
\newcommand\MAMS{\emph{Memoirs \AMS} }
\newcommand\PAMS{\emph{Proc. \AMS} }
\newcommand\TAMS{\emph{Trans. \AMS} }
\newcommand\AnnMS{\emph{Ann. Math. Statist.} }
\newcommand\AnnPr{\emph{Ann. Probab.} }
\newcommand\CPC{\emph{Combin. Probab. Comput.} }
\newcommand\JMAA{\emph{J. Math. Anal. Appl.} }
\newcommand\RSA{\emph{Random Struct. Alg.} }
\newcommand\SPA{\emph{Stoch. Proc. Appl.} }
\newcommand\ZW{\emph{Z. Wahrsch. Verw. Gebiete} }
\newcommand\DMTCS{\jour{Discr. Math. Theor. Comput. Sci.} }

\newcommand\AMS{Amer. Math. Soc.}
\newcommand\Springer{Springer}
\newcommand\Wiley{Wiley}

\newcommand\vol{\textbf}
\newcommand\jour{\emph}
\newcommand\book{\emph}
\newcommand\inbook{\emph}
\def\no#1#2,{\unskip#2, no. #1,} 

\newcommand\webcite[1]{\hfil\penalty0\texttt{\def~{\~{}}#1}\hfill\hfill}
\newcommand\webcitesvante{\webcite{http://www.math.uu.se/\~{}svante/papers/}}
\newcommand\arxiv[2][\relax]{\webcite{arXiv:#2 \testtom{#1}}}
\newcommand\testtom[1]{%
\def\xxa{#1}\def\xxb{\relax}\ifx\xxa\xxb \else [#1]\fi}

\def\nobibitem#1\par{}

\end{document}